\documentclass{amsart}

\usepackage[all]{xy}
\usepackage{hyperref}
\usepackage[capitalize]{cleveref}
\xyoption{curve}
\usepackage{longtable,booktabs}
\usepackage{xcolor}
\usepackage[pdf]{pstricks}
\usepackage{pst-plot}
\usepackage{lscape}
\usepackage{amssymb}

\newtheorem{thm}{Theorem}[section]
\newtheorem{lemma}[thm]{Lemma}
\newtheorem{prop}[thm]{Proposition}

\theoremstyle{definition} 
\newtheorem{remark}[thm]{Remark}

\newcommand{\C}{\mathbb{C}}	
\newcommand{\F}{\mathbb{F}}

\newcommand{\tmf}{\mathrm{tmf}}
\newcommand{\mmf}{\mathrm{mmf}}
\DeclareMathOperator{\Sq}{Sq}

\newcommand{\cA}{\mathcal{A}}
\newcommand{\cl}{\mathrm{cl}}

\begin{document}

\title
[Sparse $\C$-motivic family]
{A sparse periodic family in the cohomology of the $\C$-motivic Steenrod algebra}

\author[Isaksen]{Dan Isaksen}
\address{Wayne State University}
\email{isaksen.dan@gmail.com}

\author[Kong]{Hana Jia Kong}
\address{School of Mathematical Sciences, Zhejiang University, Hangzhou, China}\email{hjkong@zju.edu.cn}

\author[Li]{Guchuan Li}
\address{School of Mathematical Sciences, Peking University, Beijing, China}\email{liguchuan@math.pku.edu.cn}

\author[Ruan]{Yangyang Ruan}
\address{Beijing Key Laboratory of Topological Statistics and Applications for Complex Systems, Beijing Institute of Mathematical Sciences and Applications, Beijing 101408, China}\email{ruanyy@amss.ac.cn}

\author[Zhu]{Heyi Zhu}
\address{Department of Mathematics, University of Illinois, Urbana-Champaign, Urbana, Illinois 61801, USA}
\email{heyizhu2@illinois.edu}

\date{April 2025}

\thanks{The first author was supported by NSF grant DMS-2202267.}



\begin{abstract}
We study a particular family of elements in the cohomology of the $\C$-motivic Steenrod algebra, also known as the $\C$-motivic Adams $E_2$-page.  This family exhibits unusual periodicity properties, and it is related both to $h_1$-localization and to the algebraic Hurewicz image of the motivic modular forms spectrum $\mmf$.
\end{abstract}

\maketitle

\section{Introduction}
\label{sctn:intro}

The goal of this article is to study a particular infinite family in the cohomology of the $\C$-motivic Steenrod algebra, also known as the $E_2$-page of the $\C$-motivic Adams spectral sequence that converges to $\C$-motivic homotopy \cite{Morel99} \cite{DI10}.  Hieu Thai \cite{Tha21} first expressed interest in these specific elements and proposed a conjecture about their existence.  This manuscript contains a complete answer to Thai's conjecture.

More specifically, we study the existence of elements in degrees of the form $k (20, 4, 12) + (17, 4, 10)$.  Here we grade in terms of stem, Adams filtration, and motivic weight.  We use the symbols $e_0 g^k$ to refer to our elements because those names are consistent with their interesting properties.  Despite their ``compound'' names, these elements are in fact indecomposable in the product structure in the cohomology of the $\C$-motivic Steenrod algebra.  The indecomposability is related to the fact that $g^k$ does not survive to the $\C$-motivic Adams $E_2$-page, although $\tau g^k$ does survive.

Thai showed that $e_0 g^k$ does not exist in the cohomology of the $\C$-motivic Steenrod algebra if $k$ is not of the form $2^n-1$  \cite{Tha21}.  We prove that $e_0 g^k$ does exist if $k$ is of the form $2^n-1$.

Recall that the Chow degree of an element in degree $(s,f,w)$ is $s+f-2w$, and note that the Chow degree of $e_0 g^k$ is $1$.  In fact, the term $k(20, 4, 12)$ contributes $0$ to the Chow degree, while the term $(17, 4, 10)$ contributes $1$.

The Burklund-Xu spectral sequence \cite{BX23} \cite{Benson24} \cite{KILRZ25} is a powerful tool for studying the cohomology of the $\C$-motivic Steenrod algebra in Chow degree one.  The usefulness of this spectral sequence in Chow degree one is demonstrated in detail in \cite{KILRZ25}, and it is our main tool for studying the elements $e_0 g^k$.  We assume all of the notation and background from \cite{KILRZ25}.

\subsection*{Definition of \texorpdfstring{$e_0 g^k$}{e0gk}}

In order to make a precise statement about the existence of elements of the form $e_0 g^k$, we must be precise about what properties these elements have.  One approach is to define $e_0 g^k$ to be an element in degree $k(20, 4, 12) + (17, 4, 10)$ with the property that its $\tau$-localization in the cohomology of the classical Steenrod algebra equals the product $e_0 \cdot g^k$.  We already know (for example, by comparison to the cohomology of classical $\cA(2)$) that $e_0 g^k$ exists and is non-zero in the cohomology of the classical Steenrod algebra, so we know that there exists some $\C$-motivic element that $\tau$-localizes to it.  However, the weight of such an element is not predetermined.  One perspective on our work is that we show that the weight of such a $\C$-motivic lift of $e_0 g^k$ is $12k+10$ when $k$ is of the form $2^n-1$, and the weight of such a lift is $12k+9$ otherwise.

Another approach is to define the $\C$-motivic element $e_0 g^k$ to be an element in degree $k (20, 4, 12) + (17, 4, 10)$ with the property that it maps to the decomposable element $e_0 \cdot g^k$ in the cohomology of $\C$-motivic $\cA(2)$.  Note that $g$ is present in the cohomology of $\C$-motivic $\cA(2)$, even though it is not present in the cohomology of the $\C$-motivic Steenrod algebra.  In other words, the study of the existence of $e_0 g^k$ is relevant to the algebraic $\C$-motivic $\mmf$-Hurewicz image, i.e., the image of the map from the $\C$-motivic Adams $E_2$-page for the sphere to the $\C$-motivic Adams $E_2$-page for the $\C$-motivic modular forms spectrum $\mmf$ \cite{GIKR21}.

\subsection*{\texorpdfstring{$h_1$}{h1}-localization}

Recall that $h_1$ is not nilpotent in the cohomology of the $\C$-motivic Steenrod algebra.  We already have a thorough understanding of the $h_1$-periodic situation, both computationally \cite{GI15} and theoretically \cite{BachmannHopkins}.  However, there are some details about the $h_1$-periodic computations that have not been settled.  The $h_1$-periodicization map takes the form
\[
H^{***} \cA \rightarrow H^{***} \cA[h_1^{-1}] = \F_2[h_1^{\pm 1}][v_1^4, v_2, v_3, \ldots ],
\]
where $\cA$ is the $\C$-motivic Steenrod algebra and $H^{***} \cA$ is its tri-graded cohomology.  The target of this map is a polynomial algebra, so its elements are easy to describe.  However, we previously were not able to find preimages of specific elements in the target.  In other words, we knew the $h_1$-periodicization abstractly, but we did not understand the $h_1$-periodicization map.

Our work shows that $e_0 g^{2^n-1}$ maps to $v_{n-2}$, up to invertible $h_1$ factors.  We also know that $P h_1$ and $c_0$ map to $v_1^4$ and $v_2$ respectively, up to invertible $h_1$ factors.  By taking products of these elements, we can now easily find preimages for any element in the localization.

More generally, given an element $x$ in the localization, one might wish to find the element $y$ of $H^{***} \cA$ with smallest possible stem such that $y$ localizes to $x$.  For example, $P h_1$ and $c_0$ are the elements with smallest possible stems that localize to $v_1^4$ and $v_2$ respectively.  Our work shows that $e_0 g^{2^n-1}$ is the element with smallest possible stem that localizes to $v_{n-2}$.  Another way to express this paragraph is in terms of the quotient of $H^{***} \cA$ by the $h_1$-power torsion.  We do not have complete control over this $h_1$-free quotient, although our results here contribute new information about it.  For example, $M h_1$ in the $45$-stem is the element with smallest possible stem that localizes to $v_3^3 + v_2^2 v_4$.  However, we only know this through explicit computation, and we do not yet have a theoretical understanding of this phenomenon.

\subsection*{\texorpdfstring{$w_1$}{w1}-periodicity}

The family of elements of the form $e_0 g^{2^n-1}$ displays a curious property with respect to $w_1$-periodicity.  See \cite{And18} and \cite{KILRZ25} for the precise meaning of $w_1$-periodicity.

Each member of our family is obtained from the previous member by application of a $w_1$-periodicity operator.  More specifically,
\[
e_0 g^{2^{n+1}-1} = w_1^{2^{n+2}} \cdot e_0 g^{2^n-1}
\]
where $w_1^{2^{n+2}} \cdot (-)$ is a Massey product operator $\left\langle h_1, h_1^{2^{n+2}-1} h_{n+4}, - \right\rangle$.  However, none of these Massey product operators can be iterated, i.e.,
\[
w_1^{2^{n+2}} \cdot \left( w_1^{2^{n+2}} \cdot e_0 g^{2^n-1} \right)
\]
is not well-defined.  The elements $e_0 g^{2^n-1}$ are related by $w_1$-periodicity, but they are not a ``$w_1$-periodic family'' in the usual sense in which the degrees grow arithmetically.  Rather, these elements form some kind of ``sparse'' family in which the degrees grow geometrically.

The sparseness of the family of elements $e_0 g^{2^n-1}$ is an exotic phenomenon, in the sense that it does not occur with $v_1$-periodicity in the cohomologies of the classical or $\C$-motivic Steenrod algebras.

\subsection*{Algebraic \texorpdfstring{$\C$}{C}-motivic Hurewicz image}

The $\tmf$-Hurewicz map is the unit map $S^0 \rightarrow \tmf$ of the topological modular forms spectrum $\tmf$.  The image of this map on stable homotopy is completely known \cite{BMQ23}.  The $\C$-motivic analogue of this question is to compute the image in stable homotopy of the $\C$-motivic unit map $S^{0,0} \rightarrow \mmf$, where $\mmf$ is the $\C$-motivic modular forms spectrum \cite{GIKR21}.  The image of the $\mmf$-Hurewicz map has not been completely determined.

An approximation to the $\mmf$-Hurewicz image is the algebraic $\mmf$-Hurewicz image, which is the image of the map of $\C$-motivic Adams $E_2$-pages induced by the unit map.  Our work demonstrates that the algebraic $\mmf$-Hurewicz image is not so easy to describe.  For example, the product $e_0 g^k$ belongs to the $\mmf$-Hurewicz image if and only if $k$ is of the form $2^n-1$.  On the other hand, the element $\tau e_0 g^k$ does belong to the algebraic $\mmf$-Hurewicz image for all $k$.

We expect that other similarly complicated patterns appear in the algebraic $\mmf$-Hurewicz image, although we have not conducted a thorough analysis.

\subsection*{Notation}

We borrow all notation, terminology, and grading conventions from \cite{KILRZ25}.  The reader is especially warned about our non-standard grading conventions in the Burklund-Xu spectral sequence.  See \cite[Section 5.2]{KILRZ25} for a careful treatment.

\section{Analysis of the Burklund-Xu spectral sequence}

\begin{lemma}
\label{lem:BX-h0-local}
The $h_0$-localization of the Burklund-Xu spectral sequence in Chow degree one  takes the form
\[
E_1 = q_0 \cdot \F_2[h_0^{\pm 1}] \oplus q_1 \cdot \F_2[h_0^{\pm 1}] \oplus q_2 \cdot \F_2[h_0^{\pm 1}] \oplus \cdots.
\]
The only differential is $d_1(q_1) = q_0 \cdot h_0$, and the $E_\infty$-page takes the form
\[
E_\infty = q_2 \cdot \F_2[h_0^{\pm 1}] \oplus q_3 \cdot \F_2[h_0^{\pm 1}] \oplus \cdots.
\]
\end{lemma}

\begin{proof}
The description of the $E_1$-page follows from \cite[Section 5.2]{KILRZ25} and the observation that $\F_2[h_0^{\pm 1}]$ is the $h_0$-localization of the cohomology of the classical Steenrod algebra.  The differential $d_1(q_1) = q_0 \cdot h_0$ is established in \cite[Lemma 4.10]{BX23}.  For degree reasons, there are no other possible differentials.
\end{proof}

We restate a classical result of Adams about the vanishing of the Adams $E_2$-page above a line of slope $\frac{1}{2}$.  As in \cite[Section 2]{KILRZ25}, $s-2f$ is the $v_1$-intercept of an element in stem $s$ and Adams filtration $f$.  On a standard Adams chart, the $v_1$-intercept is the $x$-intercept of the line of slope $\frac{1}{2}$ that passes through the point $(s,f)$.  The use of $v_1$ in the terminology reflects that the slope of $v_1$-multiplication is $\frac{1}{2}$.

\begin{prop}
\label{prop:Adams-vanishing}
\cite{Ada66}
The classical Adams $E_2$-page is concentrated in degrees with $v_1$-intercept greater than or equal to $-3$, with the exception of the elements $h_0^k$ in degree $(0,k)$ for $k \geq 2$.
\end{prop}

On a standard Adams chart, \cref{prop:Adams-vanishing} says that the chart is concentrated below a line of slope $\frac{1}{2}$ and with $x$-intercept $-3$, with the exception of the elements $h_0^k$ for $k \geq 2$.

\begin{lemma}
\label{lem:AdamsE2=-3}
The only elements of the classical Adams $E_2$-page with $v_1$-intercept $-3$ are of the form $P^k h_1^3$ for $k \geq 0$.
\end{lemma}

\begin{proof}
The degrees under consideration lie within the region in which the Adams periodicity operator $P$ of degree $(8,4)$ is an isomorphism \cite{Ada66}.  The result follows from explicit low-dimensional computations in degrees up to $(11,7)$.
\end{proof}

\begin{longtable}{ll}
\caption{Elements in the Burklund-Xu $E_1$-page in Chow degree one with $v_1$-intercept equal to $-4$
\label{table:BX-E1}
} \\
\toprule
$(s,f)$ & element \\
\midrule \endfirsthead
\caption[]{Some elements of the Burklund-Xu $E_1$-page in Chow degree one with $v_1$-intercepts equal to $-4$}\\
\toprule
$(s,f)$ & element \\
\midrule \endhead
\bottomrule \endfoot
$(0,2)$ & $q_0 \cdot h_0$ \\
$(8k + 2, 4k + 3)$ & $q_0 \cdot P^k h_1^2$ \\
$(8k + 4, 4k + 4)$ & $q_1 \cdot P^k h_1^3$ \\
\end{longtable}

\begin{lemma}
\label{prop:v1-intercept=-4}
\cref{table:BX-E1} lists all non-zero elements in the Burklund-Xu $E_1$-page in Chow degree one with $v_1$-intercept equal to $-4$.
\end{lemma}

\begin{proof}
We are searching for expressions of the form $q_n \cdot x$ with $x$ in the cohomology of the classical Steenrod algebra, such that $q_n \cdot x$ has $v_1$-intercept $-4$.  Since $q_n$ has $v_1$-intercept $2^n-3$, our desired $x$ has $v_1$-intercept $-1 - 2^n$.

If $n \geq 2$, then $ -1 - 2^n \leq -5$.  By \cref{prop:Adams-vanishing}, $x$ is of the form $h_0^j$.  However, $h_0^j$ has $v_1$-intercept $-2j$, which cannot equal $-1 - 2^n$ when $n \geq 2$.

It remains to consider elements of the form $q_0 \cdot x$ and $q_1 \cdot x$.  Since $q_0$ and $q_1$ have $v_1$-intercepts $-2$ and $-1$ respectively, we are looking for $x$ with $v_1$-intercepts $-2$ or $-3$.  Then  \cite[Lemma 6.4]{KILRZ25} and \cref{lem:AdamsE2=-3} tell us the possible values of $x$.
\end{proof}

\begin{lemma}
\label{prop:h0^j.vn}
For $n \geq 3$, the elements $q_n \cdot h_0^{2^{n-1}-1}$ are non-zero permanent cycles in the Burklund-Xu spectral sequence in Chow degree one.
\end{lemma}

\begin{proof}
The elements under consideration have degree $(2^n-1, 2^{n-1})$.  If an element supported a differential, then the value of such a differential would have degree $(2^n-2, 2^{n-1} + 1)$, and it would have $v_1$-intercept 
\[
2^n - 2 - 2 (2^{n-1} + 1) = -4.
\]
\cref{prop:v1-intercept=-4} classifies such elements.  Of those elements, none lie in a degree of the form $(2^n-2, 2^{n-1} + 1)$ for $n \geq 3$.  Therefore, there are no possible targets for differentials on the elements $q_n \cdot h_0^{2^{n-1}-1}$, and they are permanent cycles.

It remains to show that the elements $q_n \cdot h_0^{2^{n-1}-1}$ are not hit by any differentials.  This follows from $h_0$-localization.  If there were such a differential, then the $h_0$-localized spectral sequence would exhibit a differential.  \cref{lem:BX-h0-local} shows that there are no such $h_0$-localized differentials.
\end{proof}

\begin{remark}
\label{rem:h0^j.vn}
The condition $n \geq 3$ in \cref{prop:h0^j.vn} is necessary.  When $n = 1$ and $n = 2$, the elements $q_1$ and $q_2 \cdot h_0$ support Burklund-Xu differentials.  (When $n = 0$, the exponent on $h_0$ is not an integer.)
\end{remark}

\begin{prop}
\label{lem:e0g^k-unique}
For $n \geq 3$, there is a unique non-zero element $x_n$ in the cohomology of the $\C$-motivic Steenrod algebra in degree $(2^{n-3}-1) \cdot (20, 4, 12) + (17, 4, 10)$.  Moreover, $x_n$ is $h_1$-periodic in the sense that no power of $h_1$ annihilates $x_n$.
\end{prop}

\begin{proof}
We are considering elements that are detected in degree $(2^n-1, 2^{n-1})$ in the Burklund-Xu spectral sequence in Chow degree one.  Such elements have $v_1$-intercept equal to $-1$.  \cite[Lemma 6.5]{KILRZ25} classifies the elements with this $v_1$-intercept.

We find two elements in the Burklund-Xu $E_1$-page in the relevant degree: $q_n \cdot h_0^{2^{n-1}-1}$ and $q_0 \cdot P^{2^{n-3}-1} h_0^2 h_3$.  The latter elements are zero in the Burklund-Xu $E_\infty$-page because of the differentials
\[
d_1 \left( q_1 \cdot P^{2^{n-3}-1} h_0 h_3 \right) = q_0 \cdot P^{2^{n-3}-1} h_0^2 h_3.
\]
On the other hand, \cref{prop:h0^j.vn} shows that $q_n \cdot h_0^{2^{n-1}-1}$ is a non-zero permanent cycle.  Let $x_n$ be the unique element that is detected by this permanent cycle.  The $h_0$-localized spectral sequence analyzed in \cref{lem:BX-h0-local} tells us that $x_n$ is an $h_1$-periodic element in the cohomology of the $\C$-motivic Steenrod algebra.
\end{proof}

We briefly recall some notation.  Let $\cA$ be the $\C$-motivic Steenrod algebra, and let $\cA(2)$ be the subalgebra of $\cA$ generated by $\Sq^1$, $\Sq^2$, and $\Sq^4$.  Then $H^{***} \cA$ is the cohomology of the $\C$-motivic Steenrod algebra, also known as the $\C$-motivic Adams $E_2$-page.  Similarly, $H^{***} \cA(2)$ is the cohomology of $A(2)$, also known as the $\C$-motivic Adams $E_2$-page for the $\C$-motivic modular forms spectrum $\mmf$.

Recall from \cite{GI15} that $H^{***}\cA \left[ h_1^{-1} \right]$ is isomorphic to $\F_2 \left[ h_1^{\pm 1} \right] \left[v_1^4, v_2, v_3, \ldots \right]$.

\begin{prop}
\label{prop:e0g^k}
For $n \geq 3$, the element $x_n$ in the cohomology of the $\C$-motivic Steenrod algebra in degree $(2^{n-3}-1) \cdot (20, 4, 12) + (17, 4, 10)$ maps to $v_n$ (up to $h_1^{\pm 1}$-multiples) under the $h_1$-localization map.
\end{prop}

\begin{proof}
We know from \cref{lem:e0g^k-unique} that $x_n$ has a non-zero value under the $h_1$-localization map.  For degree reasons, $v_n$ is the only possible non-zero value.
\end{proof}

\begin{remark}
\label{rem:A(2)}
Recall that $\cA(2)$ is the subalgebra of the $\C$-motivic Steenrod algebra that is generated by $\Sq^1$, $\Sq^2$, and $\Sq^4$.  The values of the map $H^{***}\cA[h_1^{-1}] \rightarrow H^{***}\cA(2)[h_1^{-1}]$ are completely known, as stated in \cite[Conjecture 5.5]{GI15}.  
\footnote{This conjecture has been proved, using a combination of \cite[Proposition 6.4]{GI15}, \cite{AM17}, and \cite{GIKR21}.}
Recall from \cite{GI15} that $H^{***}\cA(2) \left[ h_1^{-1} \right]$ is isomorphic to $\F_2 \left[ h_1^{\pm 1} \right] \left[ v_1^4, v_2, u \right]$, and that the map $H^{***}\cA[h_1^{-1}] \rightarrow H^{***}\cA(2)[h_1^{-1}]$ takes $v_n$ to $u^{2^{n-2}-1} v_2$.
\footnote{The element $u$ was called $a_1$ in \cite{GI15}, but $u$ is consistent with other work on the cohomology of $\C$-motivic $A(2)$ \cite{Isa09} \cite{Baer}.}
\end{remark}

\begin{thm}
\label{prop:e0g^k-final}
For $n \geq 3$, the element $x_n$ in the cohomology of the $\C$-motivic Steenrod algebra in degree $(2^{n-3}-1) \cdot (20, 4, 12) + (17, 4, 10)$ maps to the product $e_0 g^{2^{n-3}-1}$ in $H^{***} \cA(2)$.
\end{thm}

\begin{proof}
Consider the diagram
\begin{equation*}
\label{eq:e0g-diagram}
\xymatrix{
H^{***} \cA \ar[d] \ar[r] \ar[d] & H^{***} \cA(2) \ar[d] \\
H^{***} \cA \left[ h_1^{-1} \right] \ar[r] & H^{***} \cA(2) \left[ h_1^{-1} \right] \\
}
\end{equation*}
in which the vertical maps are $h_1$-localizations.  The element $x_n$ in the upper left corner maps to $v_n$ in the lower left corner, according to \cref{prop:e0g^k}.  \cref{rem:A(2)} says that $v_n$ has a non-zero image in the lower right corner.  Therefore, $x_n$ has a non-zero image in $H^{***} \cA(2)$.  Finally, we inspect the possible elements in $H^{***} \cA(2)$ in the relevant tri-degree \cite{Isa09}, and we find that $e_0 g^{2^{n-3}-1}$ is the only possible value.
\end{proof}

\begin{prop}
\label{lem:e0g^k-indecomposable}
For $n \geq 3$, the element $x_n$ is multiplicatively indecomposable in $H^{***} \cA$.
\end{prop}

\begin{proof}
\cref{prop:e0g^k-final} tells us that $x_n$ is not divisible by $h_1$ because its image $e_0 g^{2^{n-3}-1}$ in $H^{***}\cA(2)$ is not divisible by $h_1$.  On the other hand, \cref{prop:e0g^k} tells us that $x_n$ maps to an indecomposable element in the $h_1$-localization.  These two claims combine to show that $x_n$ is indecomposable.
\end{proof}

\begin{remark}
Because of \cref{prop:e0g^k-final}, we prefer to rename the indecomposable element $x_n$ to $e_0 g^{2^{n-3}-1}$.  This is consistent with past practice \cite{IWX20}.  The reader must be careful that this element is indecomposable, even though its name is made from more than one symbol.
\end{remark}

\begin{remark}
\label{rem:e0g^k-tau-local}
For all $n \geq 3$, consider the image of the indecomposable element $x_n$ of $H^{***}\cA$ under the $\tau$-localization map $H^{***} \cA \rightarrow H^{**} \cA^{\cl} [\tau^{\pm 1}]$.  Here $\cA^{\cl}$ is the classical Steenrod algebra, and $H^{**} \cA^{\cl}$ is its cohomology.

Because of the commutative diagram
\begin{equation*}
\label{eq:e0g-localization-diagram}
\xymatrix{
H^{***} \cA \ar[d] \ar[r] \ar[d] & H^{***} \cA(2) \ar[d] \\
H^{**} \cA^{\cl} \left[ \tau^{\pm 1} \right] \ar[r] & H^{**} \cA^{\cl}(2) \left[ \tau^{\pm 1} \right], \\
}
\end{equation*}
we know that the image in $H^{**} \cA^{\cl} \left[ \tau^{\pm 1} \right]$ is of the form $e_0 g^{2^{n-3} - 1} + y$, where $y$ lies in the kernel of the map $H^{**} \cA^{\cl} \rightarrow H^{**} \cA^{\cl}(2)$.  Unfortunately, our techniques do not allow us to compute $y$, although it is reasonable to guess that it equals $0$ in all cases.
\end{remark}

\bibliographystyle{amsalpha}
\bibliography{bib}

\end{document}